\documentclass[twocolumn,journal]{IEEEtran}

 \usepackage{cite,comment}
\usepackage{amsthm}
\usepackage{mathrsfs}
\usepackage{graphicx}
\usepackage{graphics}
\usepackage{amssymb}
\usepackage{amsmath,mathtools}
\usepackage{color}
\usepackage{xspace}
\usepackage{algpseudocode}
\usepackage{bbm}
\usepackage{comment}
\usepackage{algorithmicx}
\usepackage[caption=false]{subfig}
\usepackage{psfrag}
\usepackage{url}
\usepackage{float}
\usepackage{caption}
\usepackage{soul}
\usepackage{algpseudocode}
\usepackage{algorithm,algorithmicx}
\usepackage{rotating}
\usepackage{chngcntr}
\usepackage{apptools}
\AtAppendix{\counterwithin{thm}{section}}
\usepackage{graphicx}
\usepackage{graphics}
\usepackage{amssymb}
\usepackage{amsmath}
\usepackage{mathtools}
\usepackage{color}
\usepackage{xspace}
\usepackage{algpseudocode}
\usepackage{bbm}
\usepackage{comment}
\usepackage{algorithmicx}
\usepackage{subfig}
\usepackage{psfrag}
\usepackage{tikz}
\usetikzlibrary{shapes,arrows}
\usetikzlibrary{positioning}

\usepackage{rotating}
\usepackage{chngcntr}
\usepackage{apptools}
\usepackage{listings}
\AtAppendix{\counterwithin{thm}{section}}
\usepackage{graphicx}
\usepackage{graphics}
\usepackage{amssymb}
\usepackage{amsmath}
\usepackage{color}
\usepackage{xspace}
\usepackage{algpseudocode}
\usepackage{bbm}
\usepackage{comment}
\usepackage{algorithmicx}
\usepackage{subfig}
\usepackage{psfrag}
\usepackage{tikz}
\usetikzlibrary{shapes,arrows}
\usetikzlibrary{positioning}

\usepackage{enumitem}
\usepackage{graphicx}               
\usepackage{amsmath,amssymb}

\usepackage{graphics}
\usepackage{amssymb}
\usepackage{amsmath}
\usepackage{color}
\usepackage{xspace}
\usepackage{algpseudocode}
\usepackage{bbm}
\usepackage{comment}
\usepackage{algorithmicx}
\usepackage{subfig}
\usepackage{psfrag}
\usepackage{soul}
\usepackage{tikz}
\usetikzlibrary{petri}
\usetikzlibrary{shapes,arrows}
\usetikzlibrary{positioning}

\usepackage{cancel}

\tikzstyle{block}=[draw opacity=0.7,line width=1.4cm]

\newcommand{\oprocendsymbol}{\hbox{$\bullet$}}
\newcommand{\oprocend}{\relax\ifmmode\else\unskip\hfill\fi\oprocendsymbol}

\newcommand{\VV}{\mathcal{V}}
\newcommand{\EE}{\mathcal{E}}
\newcommand{\GG}{\mathcal{G}}

\newcommand{\real}{{\mathbb{R}}}
\newcommand{\reals}{{\mathbb{R}}}
\newcommand{\realpositive}{{\mathbb{R}}_{>0}}
\newcommand{\realnonnegative}{{\mathbb{R}}_{\ge 0}}

\newcommand{\until}[1]{\in\{1,\dots,#1\}}

\newcommand{\solmaz}[1]{{\color{red}#1}}

\newcommand{\amir}[1]{{\color{blue}#1}}

\newcommand{\vect}[1]{\boldsymbol{\mathbf{#1}}}
\newcommand{\vectsf}[1]{\boldsymbol{\mathbf{\mathsf{#1}}}}

\newcommand{\re}[1]{\operatorname{Re}(#1)}

\newcommand{\Diag}[1]{\operatorname{Diag}(#1)}

 \newcommand{\boxend}{\hfill \ensuremath{\Box}}
 
\newcommand{\margin}[1]{\marginpar{\color{red}\tiny\ttfamily#1}}

\newtheorem{thm}{Theorem}[section]

\newtheorem{lem}{Lemma}[section]

\title{\LARGE \bf The fastest linearly converging discrete-time average consensus using buffered information}
\author{Amir-Salar Esteki, Hossein Moradian and Solmaz S. Kia, \emph{Senior Member, IEEE} 
  \thanks{The authors are with the Department of Mechanical and
    Aerospace Engineering, University of California Irvine, Irvine, CA 92697, USA {\tt\small
      \{aesteki, hmoradia, solmaz\}@uci.edu}. This work was supported by NSF, under CAREER Award ECCS-1653838.}
}

\begin{document}

\maketitle

  \begin{abstract}
  In this letter, we study the problem of accelerating reaching average consensus over connected graphs in a discrete-time communication setting. Literature has shown that consensus algorithms can be accelerated by increasing the graph connectivity or optimizing the weights agents place on the information received from their neighbors. In this letter instead of altering the communication graph, we investigate two methods that use buffered states to accelerate reaching average consensus over a given graph. In the first method, we study how convergence rate of the well-known first-order Laplacian average consensus algorithm changes with delayed feedback and obtain a sufficient condition on the ranges of delay that leads to faster convergence.  In the second proposed method, we show how average  consensus problem can be cast as a convex optimization problem and solved by first-order accelerated optimization algorithms for strongly-convex cost functions. We construct the fastest converging average consensus algorithm using the so-called Triple Momentum optimization algorithm. We demonstrate our results using an in-network linear regression problem, which is formulated as two average consensus problems.
  \end{abstract}

\textbf{keywords}: Consensus Algorithm, Accelerated Average Consensus, Delay Systems, Multi-agent systems

\section{Introduction}
In the average consensus problem the objective is to enable a group of communicating agents $\mathcal{V}=\{1,\cdots,N\}$ to arrive at the average of their local input $\mathsf{r}^i\in\real$, i.e., to obtain $\mathsf{r}^{\text{avg}}=\frac{1}{N}\sum_{j=1}^N\mathsf{r}^j$ using local interactions. The solution to this problem is of great importance in various multi-agent applications such as robot coordination~\cite{PY-RAF-KML:08}, sensor fusion~\cite{ROS-JSS:05,ROS:07,WR-UMA:17}, distributed estimation~\cite{SM:07} and  formation control~\cite{JAF-RMM:04}. In these applications, reaching fast to the average consensus is of great interest to reduce the end-to-end delays and also the convergence error caused by premature termination of the algorithm because of time constraints. 

The well-known solution for the average consensus problem is the first-order iterative algorithm 
\begin{subequations}\label{eq::consensus-orig}
\begin{align}
   & {x}^i(t_{k+1})=x^i(t_{k})-\delta\,\sum\nolimits_{j=1}^N\!\!a_{ij}(x^j(t_k)-
    x^i(t_k)),\\
    &~x^i(t_0)=\mathsf{r}^i,\quad i\in\VV.
\end{align}
\end{subequations}
where $a_{ij}$s are the adjacency weights. In this algorithm, each agent $i$ uses the agreement feedback $\sum\nolimits_{j=1}^N\!a_{ij}(x^j(t)-x^i(t))$ to derive its local agreement state $x^i$ towards $\mathsf{r}^{\text{avg}}$. When the interaction topology of the agents is a connected undirected graph, see Fig.~\ref{fig::graph},~\cite{ROS-JAF-RMM:07} shows that with a proper choice of stepsize $\delta=t_{k+1}-t_k$, executing~\eqref{eq::consensus-orig} guarantees $x^i\!\to\!\mathsf{r}^\text{avg}$, $i\in\VV$, as $k\!\to\!\infty$. Our objective in this paper is to obtain accelerated average consensus algorithms that have a provably faster convergence rate than algorithm~\eqref{eq::consensus-orig} when the same stepsize $\delta$ is used. We consider two approaches: one using outdated agreement feedback in~\eqref{eq::consensus-orig} and the other by constructing alternative algorithms using the first-order accelerated optimization algorithms for strongly convex unconstrained optimization problems.
 
\begin{figure}[t]
\centering
\begin{tikzpicture}[scale=0.9]
\tikzset{
    nor/.style={circle,minimum size=.5cm,fill=black!20,draw},
    n1/.style={very thick,blue!65!black}}
\node[nor] (1) at (0,0) {1};
\node[nor] (2) at (2,0) {2};
\node[nor] (3) at (2,-2) {3};
\node[nor] (4) at (-2,-2) {4};
\node[nor] (5) at (0,-2) {5};
\draw[n1][-] (1) -- (2);
\draw[n1][-] (1) -- (5);
\draw[n1][-] (2) -- (3);
\draw[n1][-] (3) -- (5);
\draw[n1][-] (4) -- (5);
\draw[n1][-] (1) -- (4);
\draw[n1][-] (2) -- (5);
\end{tikzpicture}
\caption{{\small A connected undirected graph $\mathcal{G}(\VV,\mathcal{E},\vectsf{A})$ with five agents. The adjacency weights of the agents here are $a_{ij}=a_{ji}=1$ if $(i,j)$ is in the edge set $\mathcal{E}$, and $a_{ij}=a_{ji}=0$  otherwise.}}\label{fig::graph}
\end{figure}
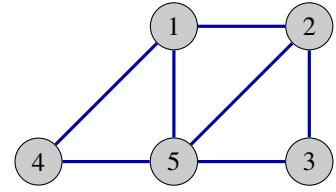

For a multi-agent system with connected  undirected communication graph, the convergence rate of the average consensus algorithm~\eqref{eq::consensus-orig} is tied to the connectivity of the graph~\cite{MF:73} and the spectral radius of matrix $(\vect{I}-\delta \vect{L})$ where $\vect{L}$ is the Laplacian matrix of the graph~\cite{ROS-RMM:04,SSK-BVS-JC-RAF-KML-SM:19}.
Given this connection, 
various studies such as optimal adjacency weight selection for a given topology by maximizing the smallest non-zero eigenvalue of the Laplacian matrix~\cite{LX-SB:04,SB-AG-BP-DS:06} or rewiring the graph to create topologies such as small-world network~\cite{SK-JMF:06,PH-JSB-VG:08} with high connectivity are proposed in the literature. In this letter, instead of altering the communication graph, we investigate two methods that use buffered states to accelerate reaching average consensus over a given graph. 

Our first method is motivated by evidence in the literature on the positive effect of time delay on increasing stability margin and rate of convergence of continuous-time linear time-delayed systems and use of delayed agreement feedback to accelerate continuous-time Laplacian average consensus algorithm, see~\cite{HM-SSK:22tcns,HM-SSK:20tac,YG-MS-CS-NM:16,ZM-YC-WR:10,YC-WR:10, WQ-RS:13}. Since the results obtained for the continuous-time Laplacian algorithm cannot be trivially extended to discrete-time communication setting, we investigate using out-dated feedback in~\eqref{eq::consensus-orig} to increase the convergence rate. More precisely, we explore for what values of non-zero $d$ in  
\begin{align}\label{eq::consensus-orig_dated}
   & {x}^i(k+1)=x^i(k)-\delta\,\sum\nolimits_{j=1}^N\!a_{ij}(x^j(t-d)-
    x^i(t-d)),\nonumber\\
   &\qquad\vect{x}(k)=0~\text{for}~ k\in\{-d ,\cdots,-1\},~x^i(0)=\mathsf{r}^i,
\end{align}
$i\in\VV$, we can archive faster convergence than~\eqref{eq::consensus-orig}. Our contribution is to characterize fully the range of delay $d$ for which convergence is accelerated.
 
Even though our results show that for appropriate values of delay $d\in\mathbb{Z}_{>0}$ in~\eqref{eq::consensus-orig_dated}, convergence can be accelerated, this method has its own limitations because of restricting the structure of the algorithm to the first-order form of algorithm~\eqref{eq::consensus-orig}. This leaves one to wonder whether faster convergence can be achieved by using alternative forms. With such motivation, for example, \cite{TCA-BNO-MJC:08} proposes to improve the rate of convergence by predicting future state values using a weighted summation of current and previous states, denoted as the mixing parameter. However, it requires a complex parameter design procedure since significant improvements in the rate of convergence are usually achieved by values outside the identified range of the mixing parameter which also requires agents to know extra global information~\cite[Equation (13)]{TCA-BNO-MJC:08}. A simple and more effective approach, however, is reported in~\cite{bu2018accelerated}, which casts average consensus problem as a convex optimization problem and uses the  accelerated  Nesterov's optimization method to design a fast-converging average consensus algorithm. Nesterov algorithms~\cite{nesterov2013introductory}, one for convex  (denoted here as NAG-C) and one for strongly-convex (denoted here as NAG-SC) cost functions, are  gradient-based optimization methods that use the buffered one-step past gradient value to accelerate convergence.  By casting the consensus algorithm as an optimization problem with the cost $\frac{1}{2}\vect{x}^\top\vect{L}\vect{x}$ where $\vect{x}$ is the aggregated agreement state of the agents,~\cite{bu2018accelerated} invokes NAG-C method to design its accelerated algorithm. The choice of NAG-C is because $\vect{L}$ of connected graphs is positive semi-definite and thus $\frac{1}{2}\vect{x}^\top\vect{L}\vect{x}$ is a convex function. In this letter, we show that with an alternative modeling approach, we can in fact use the NAG-SC to arrive at a faster converging average consensus algorithm. Our approach also opens the door for use of the so-called Triple Momentum (hereafter denoted as TM) algorithm which is the fastest known globally convergent
gradient-based method for minimizing
strongly convex functions~\cite{BVS-RAF-KML:18}. TM also uses the buffered one-step past gradient value, but has a provably faster convergence than the Nesterov algorithms.

\emph{Organization}: 
Notations and preliminaries including a brief review of the relevant properties of the time-delayed discrete-time systems and the graph theoretic definitions are given in Section~\ref{sec::prelim}.  Problem definitions and the objective statements are given in Section~\ref{sec::Prob_formu}, while the main results are given in~Sections~\ref{sec::main} and \ref{sec::acc}. Numerical simulations to illustrate our results are given in Section~\ref{sec::Num_ex}.   Section~\ref{sec::Con} summarizes our concluding remarks.
\section{Notations and Preliminaries}\label{sec::prelim}
We let $\reals$, $\realpositive$, $\realnonnegative$, $\mathbb{Z}$, and $\mathbb{C}$
denote the set of real, positive real, non-negative real, integer, and complex numbers, respectively. 
The transpose of a matrix $\vect{A}\in\real^{n\times m}$ is~$\vect{A}^\top$.
The set of eigenvalues of matrix $\vect{A}\in\real^{n\times n}$ is $\sigma(\vect{A})$ and its spectral radius is $\rho(\vect{A})$. Recall that for a square matrix $\vect{A}$, we have~\cite{RAH-CRJ:90}
\begin{align}\label{eq::spectral_formula}
   \lim_{k\to\infty}\|\vect{A}^k\|^{1/k}=\rho(\vect{A}),
\end{align}
and when $\vect{A}$ is symmetric, we have $\|\vect{A}\|=\rho(\vect{A})$.
We follow~\cite{FB-JC-SM:09} to define our graph related terminologies and notations. In a network of $N$ agents  with undirected connected graph
topology the graph is denoted by $\mathcal{G}(\mathcal{V},\mathcal{E},\vectsf{A})$ where $\mathcal{V}=\{1,\cdots,N\}$ is the node set, $\mathcal{E}\subset\mathcal{V}\times \mathcal{V}$ is the edge set and $\vectsf{A}=[a_{ij}]$ is the adjacency matrix of the graph. Recall that $a_{ii}=0$, $a_{ij}\in\real_{>0}$ if $j\in\VV$ can send information to agent $i\in\VV$, and zero otherwise. In an undirected graph the connection between the nodes is bidirectional and  $a_{ij}=a_{ji}$ if $(i,j)\in\mathcal{E}$. Finally, an undirected graph is connected if there is a path from every agent to every other agent in the network (see e.g.~Fig.~\ref{fig::graph}). The Laplacian matrix of the graph is  $\vect{L}=\text{Diag}(\vectsf{A}\vect{1}_N)-\vectsf{A}$. The Laplacian matrix of a connected undirected graph is a symmetric positive semi-definite matrix that has a simple $\lambda_1=0$ eigenvalue, and the rest of its eigenvalues satisfy $\lambda_1=0< \lambda_2\leq\cdots\leq\lambda_N$.  Moreover, $\vect{L}\vect{1}_N=\vect{0}$. 


Consider the time-delayed discrete-time system  
\begin{align}\label{eq::delay_difference}
\vect{x}(k+1)&=\vect{x}(k)+\delta\vect{A}x(k-d)\nonumber\\
\vect{x}(k)&=\vect{\psi}(k),\quad k\in\{-d ,\cdots,0\},
\end{align}
where $\vect{x}(k)\in\real^n$ is the state variable at time step $k$, $\vect{A}\in\real^{n\times n}$ is the system matrix, $\vect{\psi}(k)$ is a $\mathbb{Z}\rightarrow\reals$ function, and $d\in\mathbb{Z}_{>0}$ denotes delay. The asymptotic stability of system~\eqref{eq::delay_difference} can be assessed by the roots of its characteristic equation $\mathcal{T}:\mathbb{C}\to\mathbb{C}$ described by
\begin{align}\label{eq::char_dis}
\mathcal{T}(s)=\text{det}\left((s^{d+1}-s^d)\vect{I}-\delta\vect{A}\right).   
\end{align}
The next theorem provides the stability condition of~\eqref{eq::delay_difference}.
\begin{thm}[Stability condition for system~\eqref{eq::delay_difference}~\cite{SN:01}]\label{thm::rate_discrete} \rm For any $d\in\real_{\geq0}$, let $\{s_i\}_{i=1}^{(d+1)n}$ be the roots of the characteristic equation~\eqref{eq::char_dis}. 
Then, system~\eqref{eq::delay_difference} is asymptotically stable if and only if  $|s_i|<1, i\in\{1,\cdots,(d+1)n\}$, i.e., all the roots are located strictly inside the unit circle in the complex~plane.\boxend
\end{thm}
For $d=0$, \cite[Lemma~S1]{SSK-BVS-JC-RAF-KML-SM:19} shows that asymptotic stability of~\eqref{eq::delay_difference} is guaranteed for any 
\begin{equation}\label{eq::delta_rang}
\delta\in\left(0,\min\left\{-2\frac{\re{\mu_i}}{|\mu_i|^2}\right\}_{i=1}^n\right),
\end{equation} 
where $\{\mu_i\}_{i=1}^n$ are eigenvalues of $\vect{A}$. In order to analyze the convergence rate of~\eqref{eq::delay_difference} for any $d\in\mathbb{Z}_{\geq0}$, we define the asymptotic  convergence factor of  system~\eqref{eq::delay_difference} as 
\begin{align}\label{eq::asym_rate}
\mathsf{r}_d=\underset{\vect{x}(0)\neq\vect{0}}{\sup}\lim_{k\to\infty}\left(\frac{\|\vect{x}(k)\|}{\|\vect{x}(0)\|}\right)^{\frac{1}{k}}.
\end{align}
The associated convergence time of system~\eqref{eq::delay_difference} is then 
\begin{align}
\mathsf{t}_{d}=\frac{1}{\textup{log}(1/\mathsf{r}_d)},
\end{align}
which represents the (asymptotic) number of steps in which the norm of trajectories of~\eqref{eq::delay_difference} reduces by the factor $1/\text{e}$. For the special case,  $d=0$, since $\vect{x}(k)=(\vect{I}+\delta\vect{A})^k\vect{x}(0)$ then it follows from~\eqref{eq::spectral_formula} that
\begin{align}\label{eq::r0}
    \mathsf{r}_0=\rho(\vect{I}+\delta\vect{A}).
\end{align}


\section{Problem Definition}\label{sec::Prob_formu}
This paper considers the static average consensus problem over a connected undirected graph $G(\VV,\EE,\vectsf{A})$. Algorithm~\eqref{eq::consensus-orig} is the well-known solution for the average consensus problem. According to~\eqref{eq::delta_rang}, the admissible stepsize for algorithm~\eqref{eq::consensus-orig} over a connected graph should satisfy $\delta\in(0,\frac{2}{\lambda_N})$ so that 
algorithm~\eqref{eq::consensus-orig} 
converges exponentially fast to the average of the initial conditions of the agents~\cite{ROS-RMM:04}. 
The convergence factor of average consensus according to~\eqref{eq::consensus-orig_dated} is 
$\mathsf{r}_0=\max\{|1-\delta\lambda_2|,|1-\delta\lambda_N|\}$. For  $\delta\in(0,\frac{1}{\lambda_N}]$, given that $0<\lambda_2\leq\lambda_N$, $\mathsf{r}_0=|1-\delta\lambda_2|$. 
Given $\delta\in(0,\frac{1}{\lambda_N})$, if one wants to increase the rate of convergence of algorithm~\eqref{eq::consensus-orig} then the only possible mechanism is to decrease $\delta$, or in another word, increase the frequency of the communication between the agents. The objective in this paper is to investigate algorithms that can have provably faster convergence but using the same stepsize $\delta$. Our first approach is to investigate using out-dated feedback in~\eqref{eq::consensus-orig}, i.e., using non-zero $d$ in~\eqref{eq::consensus-orig_dated}.
Our second approach is to cast the average consensus problem as a convex optimization problem and then seek faster converging algorithms using the first-order accelerated optimization~algorithms. 

\section{Accelerated average algorithm via outdated agreement feedback}\label{sec::main}
In this section, we study convergence of algorithm~\eqref{eq::consensus-orig_dated} and determine for what values of $d\in\mathbb{Z}_{>0}$, this algorithm can have faster convergence than algorithm~\eqref{eq::consensus-orig}. For convenience in our study, we implement the change of variable  $\vect{z}(t)=\vect{T}^\top(\vect{x}(t)-\mathsf{x}^\text{avg}(0)\vect{1}_N)$ to write~\eqref{eq::consensus-orig} in the following equivalent form
\begin{subequations}\label{eq::laclacian_equivalent}
\begin{align}
    {z}_1(k+1)&=z_1(k),\quad\quad\quad {z}_1(0)=\frac{1}{\sqrt{N}}\sum\nolimits_{i=1}^N\mathsf{r}^i,\label{eq::laclacian_equivalent_z1}\\
z_i(k+1)&\!=z_i(k)\!-\delta\lambda_i z_i(k-d),\label{eq::laclacian_equivalent_z2} \\
z_{i}(0)&=[\vect{T}^\top\vectsf{x}(0)]_i, \quad z_{i}(k)=0~~~ k\in\{-d,\cdots,-1\}, \nonumber
\end{align}
\end{subequations}
for $i\in\{2,\cdots,N\}$, where
\begin{align}\label{eq::T}\vect{T}=\begin{bmatrix}\frac{1}{\sqrt{N}}\vect{1}_N&\vect{R}\end{bmatrix},\quad \vect{R}=\begin{bmatrix}v_2&\cdots&v_N\end{bmatrix}
\end{align}
satisfies $\vect{T}^\top\vect{L}\vect{T}=\vect{\Lambda}=\Diag{0,\lambda_2,\cdots,\lambda_N}$. Such $\vect{T}$ always exists, since $\vect{L}$ of a connected undirected graph is a symmetric and real matrix whose normalized eigenvectors $v_1=\frac{1}{\sqrt{N}}\vect{1}_N,v_2,\cdots,v_N$ are mutually orthogonal.
 Evidently, here we have 
\begin{align}\label{eq::x_convergence}\lim_{t\to\infty}\vect{x}(t)=&\frac{1}{\sqrt{N}}\lim_{t\to\infty}z_1(t)\vect{1}_N+\lim_{t\to\infty}\vect{R}\,\vect{z}_{2:N}(t)\nonumber\\=&(\frac{1}{N}\sum\nolimits_{i=1}^N\mathsf{r}^i)\vect{1}_N+\vect{R}\lim_{t\to\infty}\vect{z}_{2:N}(t).
\end{align}
where $\vect{z}_{2:N}=(z_2,\cdots,z_N)^\top$.
Therefore, the correctness and the convergence factor of the average consensus algorithm~\eqref{eq::consensus-orig} are determined, respectively, by asymptotic stability and the worst convergence factor of the scalar dynamics in~\eqref{eq::laclacian_equivalent_z2}.

The characteristic equation of the scalar dynamics in~\eqref{eq::laclacian_equivalent_z2} is given  by
\begin{align}\label{eq::char-lambda-i}
    \mathcal{T}(s)=s^{d+1}-s^d+\delta \lambda_i,\quad i\in\{2,\cdots,N\}.
\end{align}
The roots of~\eqref{eq::char-lambda-i} are all simple except when $\lambda_i=\frac{d^d}{\delta(d+1)^{d+1}}$ \cite{SAK:94}.
It is shown in~\cite{ISL:05} that the the roots of the characteristic equation~\eqref{eq::char-lambda-i} lie inside the unit circle if and only if $\delta\lambda_i$ lies inside the region of complex plane enclosed by the curve
\begin{align*}
\left\{z\in \mathbb{C}|z=2\vect{i}\sin(\frac{\phi}{2d+1})e^{\vect{i}\phi},-\frac{\pi}{2}\leq\phi\leq\frac{\pi}{2}\right\}.
\end{align*} 
Based on this observation and considering the asymptotic stability condition of Theorem~\ref{thm::rate_discrete},~\cite{HM-SSK:18} derives the admissible range of delay for the average consensus algorithm~\eqref{eq::consensus-orig} as follows.

\begin{lem}[Admissible range of $d$ for~\eqref{eq::consensus-orig} over connected undirected graphs~\cite{HM-SSK:18}]\label{lem::admissible-d}{\rm
Let $\GG$ be a connected undirected graph. Assume that $\delta\in(0,\frac{2}{\lambda_N})$. Then, for any $d\in\{1,\cdots,\bar{d}\}$, the average consensus algorithm~\eqref{eq::consensus-orig} satisfies $\lim_{t\to\infty}x^i=\mathsf{x}^{\text{avg}}(0)$, $i\until{N}$, (the algorithm converges asymptotically) if and only if 
\begin{align}\label{eq::stability con_dis}
\bar{d}\!=\!\min\Big\{d\in\mathbb{Z}_{\geq 0}\big|\,d>\hat{d},~~\hat{d}=&\frac{1}{2}\big(\frac{\pi}{2\arcsin(\frac{\delta\,\lambda_i}{2})}-1\big),~\nonumber\\
&\,\,i\in\{2,\cdots,N\}\Big\},
\end{align}
where $\{\lambda_i\}_{i=2}^N$ are the non-zero eigenvalues of $\vect{L}$. }\boxend
\end{lem}
According to the results in~\cite{HM-SSK:18}, we can show that the average consensus algorithm~\eqref{eq::consensus-orig} with $\delta\in(0,\frac{2}{\lambda_N})$ is guaranteed to tolerate at least one step delay. Given~\eqref{eq::laclacian_equivalent}, the convergence factor of the average consensus algorithm~\eqref{eq::consensus-orig} is 
\begin{align}
    \mathsf{r}_d=\max \{\mathsf{r}_{d,i}\}_{i=2}^N,
\end{align}
where $\mathsf{r}_{d,i}$ is the convergence factor of scalar dynamics~\eqref{eq::laclacian_equivalent_z2}, $i\in\{2,\cdots,N\}$.

The following result, whose proof is given in the appendix, specifies the range of delay $d>0$ in which the convergence factor of~\eqref{eq::laclacian_equivalent_z2} is less than of~\eqref{eq::laclacian_equivalent_z2} with $d=0$.

\begin{thm}[The range of delay which decreases the convergence factor of delay-free system  of~\eqref{eq::laclacian_equivalent_z2}]\label{lem::beta-dis_scalar}
{\rm Consider  system~\eqref{eq::laclacian_equivalent_z2}, the convergence factor for the solution of \eqref{eq::laclacian_equivalent_z2} is less than delay free case, i.e. $d=0$, if  
\begin{align}\label{eq::delay_bound_dis_rate}
  d<\min\left\{\frac{\textup{ln}(\frac{\delta\lambda_i}{\sqrt{(1-\delta\lambda_i)^2+1-2|1-\delta\lambda_i|\cos{\phi}}})}{\textup{ln}(|1-\delta\lambda_i|)},\frac{|1-\delta\lambda_i|}{1-|1-\delta\lambda_i|}\right\}, 
\end{align}
where $\phi\in[0,\frac{\pi}{d+1}]$ is the solution of $\frac{\sin{d\phi}}{\sin{(d+1)\phi}}=\frac{1}{|1-\delta\lambda_i|}$. Moreover, the convergence factor is a decreasing function of delay if $\frac{d^d}{(d+1)^{d+1}}<\{\delta\lambda_i\}_{i=2}^{N}$ holds.}
\end{thm}


 By virtue of~\eqref{eq::x_convergence}, Theorem~\ref{lem::beta-dis_scalar} provides the range of delay in terms of the eigenvalues of the Laplacian matrix such that the algorithm~\eqref{eq::consensus-orig_dated} is guaranteed to have lower convergence factor in comparison to the delay-free algorithm~\eqref{eq::consensus-orig}, i.e., it converges faster than~\eqref{eq::consensus-orig}. Therefore, given the topology of the network, faster convergence can be achieved by using outdated feedback for a fixed value of stepsize. We also note here that the range of $d$ in~\eqref{eq::delay_bound_dis_rate} provides the sufficient condition for a faster convergence while the exact range can be beyond that.

\section{Accelerated average algorithm via first-order accelerated optimization algorithms} \label{sec::acc}
In this section, we use the first-order accelerated optimization framework to devise accelerated average consensus algorithms that have proven faster convergence than the well-known average consensus algorithm~\eqref{eq::consensus-orig}. The work in this section is inspired by the results in~\cite{bu2018accelerated}.~\cite{bu2018accelerated} argued that the conventional average consensus algorithm~\eqref{eq::consensus-orig} can be viewed as the gradient descent algorithm with fixed stepsize $\delta\in(0,\frac{2}{\lambda_N})\subset\real_{>0}$ where the cost function is the agreement potential $f(\vect{x})=\frac{1}{2}\vect{x}^\top\vect{L}\vect{x}$. Note here that $\vect{0}\leq \nabla^2 f(\vect{x})=\vect{L}\leq \lambda_N\vect{I}$. Based on this observation and since the cost function~$f(\vect{x})=\frac{1}{2}\vect{x}^\top\vect{L}\vect{x}$ is convex,~\cite{bu2018accelerated} proposes to use the first-order accelerated NAG-C optimization algorithm  
\begin{subequations} \label{eq::ref15}
\begin{align}
    \vect{y}(k+1)&=\vect{x}(k)-\delta\nabla f(\vect{x}(k)),\\
    \vect{x}(k+1)&=\vect{y}(k+1)+\frac{k+1}{k+3}(\vect{y}(k+1)-\vect{y}(k)),
\end{align}
\end{subequations}
with $\vect{x}(0)=\vect{y}(0)=\vectsf{r}$, $\vectsf{r}=[\mathsf{r}^1,\cdots,\mathsf{r}^N]^{\top}$,  where $\nabla f(\vect{x}(k))=\vect{L}\vect{x}(k)$. The choice of coefficient $\frac{k+1}{k+3}$, which tends to one, is fundamental for the argument used by Nesterov to establish the following inverse quadratic convergence rate of $f(\vect{x}(k))-f(\vect{x}^\star)\leq O\left(\frac{1}{\delta \,k^2}\right)$, for any stepsize $0<\delta\leq 1/\lambda_N$, with the best step size being $\delta=\frac{1}{\lambda_N}$. Since the step size is dependent on $\lambda_N$, this algorithm requires all the agents to know this global piece of information.

For $\mu-$strongly convex objective $f$ with $L-$Lipschitz gradients, i.e., $\mu\vect{I}\leq \nabla^2 f(\vect{x})\leq L\vect{I}$, the NAG-SC algorithm achieves a faster linear convergence of $f(\vect{x}(k))-f(\vect{x}^\star)\leq O((1-\sqrt{\delta \mu})^k)$ for any stepsize of $\delta\in(0,\frac{1}{L}]$ with the best rate being achieved at $\delta=\frac{1}{L}$. The fastest accelerated globally convergent gradient-based algorithm for $\mu-$strongly convex objective $f$ with $L-$Lipschitz gradients however, is the TM algorithm proposed in~\cite{BVS-RAF-KML:18}, which achieves $f(\vect{x}(k))-f(\vect{x}^\star)\leq O\left((1-\sqrt{\frac{\mu}{L}})^{2k}\right)$. Building on the structure of these two optimization algorithms, we propose the following TM-based accelerated average consensus algorithm
\begin{subequations} \label{eq::TM}
\begin{align}
    \vect{\xi}(k+1)&=(1+\beta)\vect{\xi}(k)-\beta\vect{\xi}(k-1)-\alpha\vect{L}\vect{y}(k),\\
    \vect{y}(k)&=(1+\gamma)\vect{\xi}(k)-\gamma\vect{\xi}(k-1),\\
    \vect{x}(k)&=(1+\delta)\vect{\xi}(k)-\delta\vect{\xi}(k-1),
\end{align}
\end{subequations}
where $(\alpha,\beta,\gamma,\delta)=\Big(\frac{1+\rho}{\lambda_{N}},\frac{\rho^2}{2-\rho},\frac{\rho^2}{(1+\rho)(2-\rho)},\frac{\rho^2}{1-\rho^2}\Big)$, $\rho~=~1-\sqrt{\frac{\lambda_{N}}{\lambda_{2}}}$, $\vect{\xi}(0)=\vect{\xi}(1)=\vectsf{r}$, and the NAG-SC-based  accelerated average consensus algorithm
\begin{subequations} \label{eq::NAG-SC}
\begin{align}
    \vect{x}(k+1)&=\vect{y}(k)-\alpha\vect{L}\vect{y}(k),\\
    \vect{y}(k)&=(1+\beta)\vect{x}(k)-\beta\vect{x}(k-1),
\end{align}
\end{subequations}
with $(\alpha,\beta)=\Big(\frac{1}{\lambda_{N}},\frac{\sqrt{\lambda_{N}}-\sqrt{\lambda_{2}}}{\sqrt{\lambda_{N}}+\sqrt{\lambda_{2}}}\Big)$, $\vect{x}(0)=\vect{x}(1)=\vectsf{r}$. 
In the following, we prove the convergence of $x^i\to\frac{1}{N}\sum_{i=1}^{N}\mathsf{r}^i$ as $k\to\infty$ for the TM-based algorithm~\eqref{eq::TM}, and omit the proof for the NAG-SC algorithm, since a similar approach can be applied.

Consider the change of variable $\begin{bmatrix}w_{1} \quad \vect{w}^{\top}_{2:N}\end{bmatrix}^{\top}=\vect{T}^\top\vect{\xi}$, $\begin{bmatrix}q_{1} \quad \vect{q}^{\top}_{2:N}\end{bmatrix}^{\top}=\vect{T}^\top\vect{y}, \begin{bmatrix}p_{1} \quad \vect{p}^{\top}_{2:N}\end{bmatrix}^{\top}=\vect{T}^\top\vect{x},$
where $\vect{T}$ is defined in~\eqref{eq::T}. Then,~\eqref{eq::TM} can be written in an equivalent form
\begin{subequations}\label{eq::TMequ}
\begin{align}
    w_{1}(k+1)&=(1+\beta)w_{1}(k)-\beta w_{1}(k-1), \label{eq::TMequ-a}\\
    q_1(k)&=(1+\gamma)w_1(k)-\gamma w_1(k-1), \label{eq::TMequ-b} \\
    p_1(k)&=(1+\delta)w_1(k)-\delta w_1(k-1), \label{eq::TMequ-c} \\
    \vect{w}_{2:N}(k+1)&=(1+\beta)\vect{w}_{2:N}(k)-\beta \vect{w}_{2:N}(k-1) \nonumber \\
    &-\alpha\vect{L}^{+}\vect{y}_{2:N}(k), \label{eq::TMequ-d}\\
    \vect{p}_{2:N}(k)&=(1+\gamma)\vect{w}_{2:N}(k)-\gamma \vect{w}_{2:N}(k-1), \label{eq::TMequ-e} \\
    \vect{q}_{2:N}(k)&=(1+\delta)\vect{w}_{2:N}(k)-\delta \vect{w}_{2:N}(k-1), \label{eq::TMequ-f},
\end{align}
\end{subequations}
where $\vect{L}^{+}=\vect{R}^{\top}\vect{L}\vect{R}$. The following theorem proves that~\eqref{eq::TM} is a solution for the average consensus problem. The reason that we can use the TM and NAG-SC algorithms to design our accelerated average consensus algorithms reveals itself in the proof this theorem.  
\begin{thm}\label{thm::TM_based_Convereg}
    Consider a network of $N$ agents communicating over a connected graph. Let the agents of the network implement algorithm~\eqref{eq::TM}. Then, $x^i\to\frac{1}{N}\sum_{i=1}^{N}\mathsf{r}^i$ as $k\to\infty$.
\end{thm}
\begin{proof}
Let us consider the equivalent form of the TM algorithm in~\eqref{eq::TMequ}. From \eqref{eq::TMequ-a}-\eqref{eq::TMequ-c}, it is trivially concluded that $q_1(k)=p_1(k)=w_1(0)$ for $k\in\mathbb{Z}_{>0}$. On the other hand, since $\vect{L}^{+}$ is a positive definite matrix with eigenvalues $\lambda_2,\cdots,\lambda_N$, by virtue of the TM algorithm of~\cite{bu2018accelerated}, \eqref{eq::TMequ-d}-\eqref{eq::TMequ-f} minimize the $\lambda_2$-strongly convex function $f(\vect{p}_{2:N})=\frac{1}{2}\vect{p}_{2:N}\vect{L}^{+}\vect{p}_{2:N}$ with $\lambda_N$-Lipschitz gradient to the optimal point $\vect{p}^{\star}_{2:N}=\vect{0}$ with a rate of convergence of $f(\vect{p}_{2:N}(k))-f(\vect{0})\leq O\left(\left(1-\sqrt{\frac{\lambda_2}{\lambda_N}}\right)^{2k}\right)$. Therefore, as $k\to\infty$, $\vect{p}_{2:N}\to\vect{0}$. Considering the change of variables, we know that $q_1(0)=\frac{1}{N}\sum_{i=1}^{N}\mathsf{r}^i$ and thus, $x^i\to\frac{1}{N}\sum_{i=1}^{N}\mathsf{r}^i$ as $k\to\infty$. 
\end{proof}

We showed in the result above that the algorithm presented in~\eqref{eq::TM} solves the average consensus problem. A Similar result can be induced for~\eqref{eq::NAG-SC}. Based on the developments in~\cite{bu2018accelerated}, it is proved that \eqref{eq::ref15} converges to the average of local reference values asymptotically and faster than the popular solution~\eqref{eq::consensus-orig}. Moreover, from~\cite{BVS-RAF-KML:18}, we know that the TM algorithm benefits the fastest exponential convergence rate among other accelerated first-order gradient methods in optimization, e.g.,~\eqref{eq::ref15}. We illustrate this comparison in a numerical example in the following section with also simulating the algorithm in~\cite{TCA-BNO-MJC:08}. However, similar to~\cite{TCA-BNO-MJC:08}, we note that faster convergence in~\eqref{eq::TM} and \eqref{eq::NAG-SC} comes with the trade-off of requiring the agents to know $\lambda_2$ and $\lambda_N$ globally in order to compute $(\alpha,\beta,\gamma,\delta)$.
This trade-off also applies to~\eqref{eq::ref15} where $\lambda_N$, as a global information, is essential to compute $\delta$ locally. Therefore, by having extra global information, faster convergence can be achieved using the presented accelerated solutions.
\section{Numerical Examples}\label{sec::Num_ex}
In this section, we investigate the numerical examples that show the effect of outdated feedback data and the implementation of accelerated first-order optimization methods. We use real data from~\cite{utts2021mind} to solve a linear regression problem by reformulating it as two average consensus problems. In the first example, the common Laplacian average consensus algorithm is compared with the proposed algorithm in~\eqref{eq::consensus-orig_dated} with different number of delays to analyze the effect of $d$ on convergence. In the second example, convergence of the proposed accelerated algorithms~\eqref{eq::ref15}-\eqref{eq::NAG-SC} are compared against \eqref{eq::consensus-orig_dated} and the algorithm in~\cite{TCA-BNO-MJC:08}.

The dataset in~\cite{utts2021mind} has a size of 50 points for the 50 states in the United States. The variables are $y$ which is year 2002 birth rate per 1000 females 15 to 17 years old, and $x$ which is the poverty rate (the percent of the state’s population living in households with incomes below the federally defined poverty level). The objective is to find a linear relation between $x$ and $y$, i.e., solve the following problem:
\begin{align}\label{eq::regression}
    \textup{min}_{a,b\in\mathbb{R}}\sum\nolimits_{i=1}^{50}\|y^i-(ax^i+b)\|^2
\end{align}
for $a$ and $b$. Since our goal is to study the effect of delay in convergence, we simplify the problem and assume to know the optimal value of $b$. Thus, we only solve for the variable $a$ as the slope of the fitted line. Here, $b=4.267$. By taking a derivative from~\eqref{eq::regression} with respect to $a$, setting it to zero and substituting the value of $b$, we conclude that
\begin{align}\label{eq::RegressionSolution}
    a=\frac{\sum_{i=1}^{50}x^i(y^i-b)}{\sum_{i=1}^{50}(x^i)^2}.
\end{align}
Let us now solve this problem while the dataset is distributed over a network of 5 agents where each agent has access to 10 arbitrary state data points. Agents of the network communicate over a connected graph depicted in Fig.~\ref{fig::graph}. This solution can be achieved by solving two average consensus problems for the nominator and the denominator, and computing the division. Let $\eta_1^i(k)$ and $\eta_2^i(k)$ denote agent $i$'s local estimate of the nominator and the denominator of~\eqref{eq::RegressionSolution}, respectively. The reference values of the first average consensus problem are $\mathsf{r}_1^1=\sum_{i=1}^{10}x^i(y^i-b)$, $\mathsf{r}_1^2=\sum_{i=11}^{20}x^i(y^i-b)$,$\cdots$, $\mathsf{r}_1^5=\sum_{i=41}^{50}x^i(y^i-b)$. The same process is applied to the second average consensus problem with the reference values $\mathsf{r}_2^1=\sum_{i=1}^{10}(x^i)^2$, $\mathsf{r}_2^2=\sum_{i=11}^{20}(x^i)^2$,$\cdots$, $\mathsf{r}_2^5=\sum_{i=41}^{50}(x^i)^2$. Each agent computes the estimate $\hat{a}^i(k)=\frac{\eta_1^i(k)}{\eta_2^i(k)}$ at every step $k\in\mathbb{Z}_{\geq0}$.

\emph{Outdated feedback}:
Let the agents implement algorithm~\eqref{eq::consensus-orig_dated} with three different values of $d=\{1,5,10\}$ and a delay-free case. Fig.~\ref{fig::error1} depicts the convergence error of agent's trajectories with respect to the optimal value in~\eqref{eq::RegressionSolution}, i.e., $e(k)=\textup{log}\sum_{i=1}^{N}(\hat{a}^i(k)-a)^2$. The Green line, representing the implementation of one step delay, as seen in the figure, converge faster than the delay-free case. By increasing $d$ and using further outdated feedback, faster convergence is achieved. This shows the effect of delay in our analysis. However, as mentioned previously, exceeding $\bar{d}$ may result in divergence or slower convergence. The Turquoise trajectories with $d=10$ illustrate the fluctuation in convergence.

\emph{Accelerated consensus via first-order accelerated optimization algorithms}:
Next, we compare the convergence rate of the accelerated algorithms~\eqref{eq::ref15}, \eqref{eq::TM} and \eqref{eq::NAG-SC} with that of the accelerated algorithm proposed in~\cite{TCA-BNO-MJC:08} and also algorithm~\eqref{eq::consensus-orig_dated}. Let the agents of the network solve two average consensus problems to reach $a$ globally. Fig.~\ref{fig::error2} shows the convergence error trajectories $e(k)=\textup{log}\sum_{i=1}^{N}(\hat{a}^i(k)-a)^2$ reaching the agreement similar to the previous example. Algorithm~\eqref{eq::consensus-orig_dated} with $d=5$ converges slower compared to others, while the TM algorithm converges the fastest. Despite using   the optimal parameters for the algorithm of  in~\cite{TCA-BNO-MJC:08}, still TM-based algorithm delivers the fastest convergence.
\begin{figure}
    \centering
    \subfloat[]{\includegraphics[trim= 1pt 5pt 0 0 ,clip,width=.8\linewidth]{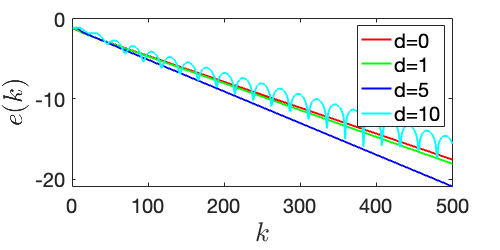}\label{fig::error1}} \vspace{-0.1in}\\
    \subfloat[]{\includegraphics[trim= 1pt 5pt 0 0 ,clip,width=.8\linewidth]{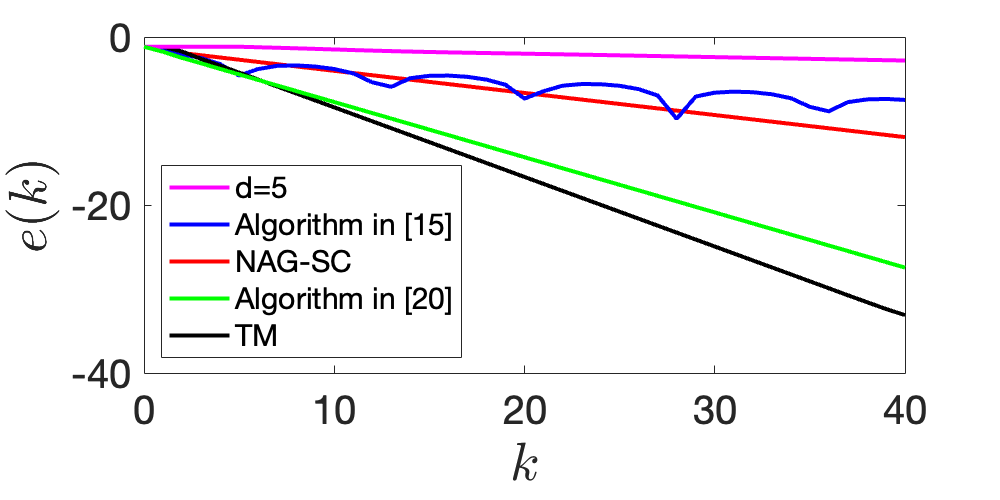}\label{fig::error2}}
    \caption{{\small Plot (a) compares the effect of different values of delay ($d=\{0,1,5,10\}$) on the convergence error $e(k)$ over $k$. We can see that $d=5$ generates the best results. Plot (b) shows the comparison between the optimization inspired algorithms~\eqref{eq::ref15}-\eqref{eq::NAG-SC} and the algorithm~\eqref{eq::consensus-orig_dated} with $d=5$. As expected from the results in section~\ref{sec::acc}, algorithm~\eqref{eq::TM} converges faster than the other methods.
    }}\vspace{-0.1in}
    \label{fig::error}
\end{figure}
\vspace{-5pt}
\section{Conclusion}\label{sec::Con}\vspace{-2pt}
We examined the effect of time delay on the rate of convergence of a discrete linear time-delayed system. We showed that for this system, we can attain higher convergence rate by adding delayed feedback. The range of delay for faster convergence was then obtained. Next, we introduced an accelerated average consensus algorithm inspired by the NAG-SC and TM algorithms. We showed that the proposed algorithms have a better convergence rate compared to a similar approach in the literature~\cite{bu2018accelerated}. To study the results of this letter, we solved a distributed linear regression problem consisted of two average consensus problems. We showed that the time-delayed system converges faster for some ranges of delay than the delay-free case. It was also indicated that the TM algorithm achieves faster convergence compared to the algorithms presented in~\cite{bu2018accelerated} and \cite{TCA-BNO-MJC:08}.
An important conclusion that can be drawn from our study is that restricting the structure of the average consensus algorithm to the Laplacian-based form~\eqref{eq::consensus-orig_dated} limits that benefits one can get from use of buffered information to achieve faster convergence. Relaxing the composition of the algorithm, allows us to arrive at faster average consensus algorithms using the well-established accelerated optimization algorithms.

\section*{Appendix}
We use the following auxiliary lemma in the proof of Theorem~\ref{lem::beta-dis_scalar}, which we present afterwards.
\begin{lem}[Location of roots of $s^{d+1}-s^d+c=0$~\cite{SAK:94}]\label{thm::dis_time}{\rm
Let $\mathcal{S}=\{s^i\in\mathbb{C}|s^{d+1}-s^d+c=0\}$ and $s^1=\max\{|s^i|\,|s^i\in\mathcal{S}\}$ for any $i\in\{1,\cdots,d+1\}$. Then for  $d\in\mathbb{Z}_{>1}$   and $c\in\real_{>0}$ all the roots    are inside the disk $|s|<\frac{1}{|a|}$ if and only if $|a|<\frac{d+1}{d}$ and
\begin{align}
    \frac{(|a|-1)}{|a|^{d+1}}<c<\frac{\sqrt{a^2+1-2|a|\cos\phi}}{|a|^{d+1}}
\end{align}
where $\phi\in[0,\frac{\pi}{d+1}]$ is the solution of $\frac{\sin(d\phi)}{\sin((d+1)\phi)}=\frac{1}{|a|}$.  

Moreover if $0<c<\frac{d^d}{(d+1)^{d+1}}$, as $c$ increases then the  value of 
$|s^1|$ decreases while the absolute value of 
all the other roots increase. In addition the smallest value of $|s^1|$   occurs at  $c=\frac{d^d}{(d+1)^{d+1}}$ where $|s^1| =\frac{d}{d+1}$.}\boxend
\end{lem}

\noindent Table.~\ref{table.1} shows different values of $\frac{d^d}{(d+1)^{d+1}}$ for a given $d\in\{1,\cdots,5\}$.

\begin{table}[]
    \centering
   \begin{tabular}{ |c|c|c|c|c|c| } 
 \hline
 $d$ & $1$ & $2$ & $3$ & $4$ & $5$\\ 
 \hline
 $\frac{d^d}{(d+1)^{d+1}}$ 
  & $0.250$  
  & $0.148$   
  & $0.105$   
  & $0.082$   
  & $0.067$  \\ 
 \hline
 \end{tabular}
    \caption{The  values of $\frac{d^d}{(d+1)^{d+1}}$ for a given $d$.
    }\vspace{-10pt}
    \label{table.1}
\end{table}

\begin{proof}[Proof of Theorem~\ref{lem::beta-dis_scalar}]
Notice that for $d=0$ the asymptotic convergence factor of \eqref{eq::laclacian_equivalent_z2} is equal to $\mathsf{r}_0=|1-\delta\lambda_i|$. Hence, our aim is to find the values of $d$ such that $\mathsf{r}_d<\mathsf{r}_0$, which means that the roots of the characteristic equation~\eqref{eq::char-lambda-i} lie inside the disk, $|s|<|1-\delta\lambda_i|$.
  Theorem~\ref{thm::dis_time} implies that it holds if and only if 
  \begin{subequations}
  \begin{align}
  &\frac{1}{|1-\delta\lambda_i|}<\frac{d+1}{d}\label{eq::dis_scal_proof_a}\\
  &\frac{|\frac{1}{1-\delta\lambda_i}|\!-\!1}{|\frac{1}{1-\delta\lambda_i}|^{d+1}}<\!\delta\lambda_i\!<\!\frac{\sqrt{\frac{1}{(1-\delta\lambda_i)^2}\!+\!1\!-\!2|\frac{1}{1-\delta\lambda_i}|\cos{\phi}}}{|\frac{1}{1-\delta\lambda_i}|^{d+1}},\label{eq::dis_scal_proof_b}
 \end{align}
  \end{subequations}
  where $\phi\in[0,\frac{\pi}{d+1}]$ is the solution of $\frac{\sin{d\phi}}{\sin{(d+1)\phi}}=\frac{1}{|1-\delta\lambda_i|}$.~By equation \eqref{eq::r0} we know $-1<1-\delta\lambda_i<1$. From~\eqref{eq::dis_scal_proof_a} we get $d<\frac{|1-\delta\lambda_i|}{1-|1-\delta\lambda_i|}$. The left-side inequality of~\eqref{eq::dis_scal_proof_b}~is satisfied for any $\delta$ with $|1-\delta\lambda_i|\!<\!1$. By some algebraic manipulation, the right-side inequality deduces $$d\!<
  \!\frac{\text{ln}(\frac{\delta\lambda_i}{\sqrt{(1-\delta\lambda_i)^2+1-2|1-\delta\lambda_i|\cos{\phi}}})}{\text{ln}(|1-\delta\lambda_i|)},$$  which concludes \eqref{eq::delay_bound_dis_rate}. The last~statement is direct application of Theorem~\ref{thm::dis_time} for~$c=\delta\lambda_i$.  
\end{proof}

\bibliographystyle{ieeetr}

\end{document}